\documentclass[11pt]{amsart}
\usepackage{amsmath, amsthm, amscd, amsfonts}
\allowdisplaybreaks 
\usepackage{tikz}
\usepackage{pstricks,pst-node}
\usepackage[all]{xy}

\makeatletter \oddsidemargin.9375in \evensidemargin \oddsidemargin
\newtheorem{theorem}{Theorem}[section]
\newtheorem{lemma}[theorem]{Lemma}
\newtheorem{proposition}[theorem]{Proposition}
\newtheorem{corollary}[theorem]{Corollary}
\theoremstyle{definition}
 
\newtheorem{example}[theorem]{Example}

\theoremstyle{remark}
\newtheorem{remark}[theorem]{Remark}
\numberwithin{equation}{section}

\begin{document}

\title[Modulation Invariant Spaces on LCA Groups ]{Modulation Invariant Spaces  on
  Locally Compact Abelian Groups}

\author[M. Mortazavizadeh, R. Raisi Tousi ]{M. Mortazavizadeh, R. Raisi Tousi*\\
May 21, 2019}



\begin{abstract}
We define and investigate modulation invariant spaces on a locally compact abelian group $G$ with respect to a closed subgroup of the dual group $\widehat{G}$. Using a range function approach, we establish a characterization of modulation invariant spaces. Finally, we define a metric on the collection of all modulation invariant spaces and study some topological properties of the metric space.
\end{abstract} \maketitle

\section{Introduction and Preliminaries}
\noindent

 For a locally compact abelian (LCA) group  $G$, a translation invariant space is defined to be a closed subspace of $L^2(G)$  that is invariant under translations by elements of a closed subgroup  $\Gamma$ of $G$. Translation invariant spaces in case of $\Gamma$ closed, discrete and cocompact, called shift invariant spaces, have been studied in \cite{  BRz, Cab, KRr, KRs, KRT, RS}, and extended to the case of $\Gamma$ closed and cocompact (but not necessarily discrete) in \cite{BR} (see also \cite{Hel, HelS}). Recently, translation invariant spaces have been generalized in \cite{BHP} to the case when $\Gamma$  is closed (not necessarily discrete or cocompact), see also \cite{I}. Another spaces, which are effective tools in Gabor theory, are spaces invariant under modulations. Due to the important role of the Gabor theory in mathematical analysis and its applications, it is important to study modulation invariant spaces.  Our main result in this paper is a characterization of modulation invariant spaces in terms of range functions. The basic idea is the fact that the image of a modulation invariant subspace of $L^2(G)$ under the Fourier transform is a translation invariant subspace of $L^2(\widehat{G})$. We use this fact to transform modulation invariant spaces to translation invariant spaces and then we follow the ideas in \cite{BR, BHP}.  By transforming $L^2(G)$ into a vector valued space, in such a way that modulations by a closed subgroup of $\widehat{G}$ become multiplications by a nice family of functions,  we characterize modulation invariant spaces in terms of range functions. Another  goal of the paper is to consider the collection of all modulation invariant spaces as a metric space. Using   range functions, we define a metric on the collection of all modulation invariant spaces and we study some topological properties of this metric space,  that is   we show that it is  complete, noncompact, and disconnected. The manuscript is organized as follows. The rest of this section is devoted to   proposing 
  some required    preliminaries related to   translation invariant spaces from \cite{BR, BHP}. Section 2 contains our main results related to modulation invariant spaces. We investigate modulation invariant spaces using a range function approach. We then define and investigate a topology (in fact a metric) on the collection of all modulation invariant spaces on $G$ in Section 3. 
 
Let $(\Omega, m)$ be a $\sigma$- finite measure space and $\mathcal{H}$ be a separable Hilbert space. A range function is a mapping $J: \Omega \longrightarrow \lbrace \textrm{  closed subspaces  of   $\mathcal{H}$ } \rbrace$. We write $P_{J}(\omega)$ for the orthogonal projection of $\mathcal{H}$ onto $J(\omega)$. A range function $J$ is measurable if the mapping $\omega \mapsto \langle P_{J}(\omega)(a) , b \rangle$ is measurable for all $a ,b\in\mathcal{H} $. Consider the space $L^{2}(\Omega, \mathcal{H})$ of all measurable functions $\phi$ from $\Omega$ to $\mathcal{H}$ such that $\Vert \phi \Vert _{2}^{2} = \int_{\Omega} \Vert \phi(\omega)\Vert^{2}_{\mathcal {H}} dm(\omega) <\infty$ with the inner product $\langle \phi , \psi \rangle = \int_{\Omega} \langle \phi(\omega) , \psi(\omega) \rangle_{\mathcal{H}} d m(\omega)$. A subset $\mathcal{D}$ of $L^{\infty}(\Omega) $ is said to be a determinig set for $L^{1}(\Omega) $, if for all $ f\in L^{1}(\Omega)$, $\int_{\Omega} fg dm =0 $ for all $ g \in \mathcal{D}$ implies that $f=0$. A closed subspace $V$ of $L^{2}(\Omega, \mathcal{H})$ is called  multiplicatively invariant  with respect to a determining set $\mathcal{D}$, if for each $\phi \in V$ and $g \in \mathcal{D}$ one has $g\phi \in V$. Bownik and Ross in \cite[Theorem 2.4]{BR} showed that there is a correnpondence between  multiplicatively invariant spaces and measurable range functions as follows.
\begin{proposition} \label{p1.1}
Suppose that $L^{2}(\Omega)$ is separable, so that $L^{2}(\Omega, \mathcal{H})$ is also separable. Then for a closed subapace $V$ of $L^{2}(\Omega, \mathcal{H})$ and a determining set $\mathcal{D}$  for $L^{1}(\Omega) $ the following are equivalent.\\
(1)  $V$ is  multiplicatively invariant with respect to $\mathcal{D}$.\\
(2)  $V$ is  multiplicatively invariant with respect to $L^{\infty}(\Omega)$.\\
(3) There exists a measureble range function $J$ such that 
\begin{equation*}
V=\lbrace \phi \in L^{2}(\Omega , \mathcal{H}) : \phi(\omega) \in J(\omega) \ \text{,} \ \  \text{ a.e. }  \omega \in \Omega \rbrace .
\end{equation*}
Identifying range functions which are equivalent almost everywhere, the correspondence between $\mathcal{D}$- multiplicatively invariant spaces and measurable range functions is one to one and onto. Moreover, there is a countable subset $\mathcal{A}$ of  $L^{2}(\Omega, \mathcal{H})$ such that $V$ is the smallest closed $\mathcal{D}$- multiplicatively invariant space containing $\mathcal{A}$. For any such $\mathcal{A}$ the measurable range function associated with $V$ satisfies
\begin{equation*}
J(\omega) = \overline{span} \lbrace \phi (\omega) : \phi \in \mathcal{A}\rbrace \ \  a.e. \  \omega \in \Omega .  \label{J} 
\end{equation*}
\end{proposition}

Now assume that $G$ is a second countable LCA group and $\Gamma$ is a closed subgroup of $G$. Let $\Gamma ^*$ be the annihilator of $\Gamma$ in $\widehat{G}$. Also suppose that $\Omega$ is a measurable section for the quotient $\widehat{G} / \Gamma ^{*}$ and $C$ is a measurable section for the quotient $G / \Gamma$. For $\gamma \in \Gamma$ we denote by $X_{\gamma}$ the associated character on $\widehat{G}$, i.e. $X_{\gamma}(\chi )= \chi(\gamma)$ for  $\chi \in \widehat{G}$. One can see that the set $ \lbrace X_{\gamma} \vert_{\Omega} : \gamma \in \Gamma \rbrace $ is a determining set for $L^{1}(\Omega) $. A closed subspace $ V \subseteq L^{2}(G)$ is called $\Gamma$- translation invariant space, if $T_{\gamma} V \subseteq V $ for all $\gamma \in \Gamma$. We say that $V$ is generated by a countable subset $\mathcal{A}$ of  $L^{2}(G)$, when $V=S^{\Gamma}(\mathcal{A})=\overline{span}\lbrace T_{\gamma}f : f \in \mathcal{A} , \gamma \in \Gamma \rbrace$.
In \cite[Proposition 6.4]{BHP} it is shown  that there exists an isometric isomorphism between $L^{2} (G)$ and $ L^{2}(\Omega , L^{2}(C))$, namely $Z: L^{2} (G) \longrightarrow L^{2}(\Omega , L^{2}(C))$  satisfying 
 \begin{equation} \label{Zak}
 Z(T_{\gamma}\phi )= X_{\gamma} \vert_{\Omega} Z (\phi).
 \end{equation}
  The forthcoming proposition, which is \cite[Theorem 6.5]{BHP}, states that $Z$ turns $\Gamma$- translation invariant  spaces in $L^{2} (G)$ into multiplicatively invariant spaces in $L^{2}(\Omega , L^{2}(C))$ with respect to the determining set  $ \mathcal{D}= \lbrace X_{\gamma} \vert_{\Omega} : \gamma \in \Gamma \rbrace$ and vice versa. It also establishes a characterization of $\Gamma$- translation invariant  spaces in terms of range functions.
\begin{proposition} \label{p1.2u}
Let $V \subseteq L^{2}(G)$ be a closed subspace and $Z$ be as in \eqref{Zak}. Then the following are equivalent.\\
(1)  $V$ is a $\Gamma$- translation invariant  space. \\
(2)  $Z(V)$ is a multiplicatively invariant subspace of $ L^{2}(\Omega , L^{2}(C))$ with respect to the determining set $ \mathcal{D}= \lbrace X_{\gamma} \vert_{\Omega} : \gamma \in \Gamma \rbrace $. \\
(3) There exists a measurable range function $J: \Omega \longrightarrow \lbrace  closed \ subspaces \  of \   L^{2}(C) \rbrace $ such that 
\begin{equation*}
V=\lbrace f \in L^{2}(G) : Z(f)(\omega)\in J(\omega) \  \text{,} \ \  \text{for a.e. }  \omega \in \Omega \rbrace.
\end{equation*}
Identifying range functions which are equivalent almost everywhere, the correspondence between $\Gamma$- translation invariant spaces and measurable range functions is one to one and onto. Moreover if $V=S^{\Gamma}(\mathcal{A})$ for some countable subset $\mathcal{A}$ of $L^{2}(G)$, the measurable range function $J$  associated with $V$ is given by
\begin{equation*}
J(\omega) = \overline{span} \lbrace Z(\phi) (\omega) : \phi \in \mathcal{A}\rbrace \ \  a.e.  \  \omega \in \Omega .  
\end{equation*}
\end{proposition}

\section{Modulation Invariant Spaces}

Our goal in this section is a characterization of modulation invariant spaces in terms of range functions. The idea is that  we transfer modulation invariant spaces to translation invariant spaces, and then we  apply  Proposition \ref{p1.2u}   to  the latter spaces, with the aid 
of which we characterize
 modulation invariant spaces in terms of range functions. Let $G$ be an LCA group and $\Lambda$ be a closed subgroup of $\widehat{G}$, which is not necessarily discrete or cocompact. A closed subspace $ W \subseteq L^{2}(G)$ is called $\Lambda$- modulation invariant, if $M_{\lambda} W \subseteq W $ for all $\lambda \in \Lambda$, where $M_{\lambda}$ is the modulation operator defined as $M_{\lambda} : L^{2} (G) \longrightarrow L^{2} (G)$, $M_{\lambda}f(x)=\lambda(x)f(x)$.  We say that $W$ is generated by a countable subset $\mathcal{A}$ of  $L^{2}(G)$, when $W=M^{\Lambda}(\mathcal{A})=\overline{span}\lbrace M_{\lambda}f : f \in \mathcal{A} , \lambda \in \Lambda \rbrace$. When $\mathcal{A} = \lbrace \varphi \rbrace $, $M^{\Lambda}(\varphi)$ is called a principal modulation invariant space. Assume that  $\Lambda^{*}$ is the annihilator of $\Lambda$ in $G$. In addition, suppose that $\Pi$ is a measurable section for the quotient $G / \Lambda ^{*}$ and $D$ is a measurable section for the quotient $\widehat{G} / \Lambda$. For $\lambda \in \Lambda$ we denote by $X_{\lambda}$ the associated character on $G$. One can see that the set $ \mathcal{D}= \lbrace X_{\lambda} \vert_{\Pi} : \lambda \in \Lambda \rbrace $ is a determining set for $L^{1}(\Pi) $. 
 Let $Z: L^2(\widehat{G}) \longrightarrow L^{2}(\Pi , L^{2}(D)) $ be similar to \eqref{Zak} and $\mathcal{F}$ be the Fourier transform on $L^2(G)$. We define an isometric isomorphism as
 \begin{equation} \label{ztild}
  \tilde{Z} : L^2(G) \longrightarrow L^{2}(\Pi , L^{2}(D))  , \  \tilde{Z}:= Z \ o\ \mathcal{F} .
 \end{equation}
In the next theorem, we show that $ \tilde{Z}$  turns $\Lambda$- modulation invariant  spaces in $L^{2} (G)$ into multiplicatively invariant spaces in $L^{2}(\Pi , L^{2}(D))$  and vice versa. Further we establish a charactrization of $\Lambda$- modulation invariant  spaces in terms of range functions. The main idea of the proof is that the Fourier transform maps $\Lambda$- modulation invariant  subspaces of $L^{2}(G)$ to $\Lambda$- translation invariant  subspaces of $L^{2}(\widehat{G})$.
\begin{theorem} \label{p1.2}
Let $W \subseteq L^{2}(G)$ be a closed subapace and $\tilde{Z}$ be as in \eqref{ztild}. Then the following are equivalent.\\
(1)  $W$ is a $\Lambda$- modulation invariant  space. \\
(2)  $\tilde{Z}(W)$ is a multiplicatively invariant subspace of $ L^{2}(\Pi , L^{2}(D))$ with respect to the determining set $ \mathcal{D}= \lbrace X_{\lambda} \vert_{\Pi} : \lambda \in \Lambda \rbrace $. \\
(3) There exists a measurable range function $J: \Pi \longrightarrow \lbrace  closed \ subspaces \  of \   L^{2}(D) \rbrace $ such that 
\begin{equation} \label{mi}
W=\lbrace f \in L^{2}(G) : \tilde{Z}(f)(x)\in J(x)  \text{,} \ \  \text{for a.e. }  x \in \Pi \rbrace .
\end{equation}
Identifying range functions which are equivalent almost everywhere, the correspondence between $\Lambda$- modulation invariant spaces and measurable range functions is one to one and onto. Moreover if $W=M^{\Lambda}(\mathcal{A})$, for some countable subset $\mathcal{A}$ of $L^{2}(G)$, the measurable range function $J$  associated with $W$ is given by
\begin{equation*}
J(x) = \overline{span} \lbrace \tilde{Z}(\phi) (x) : \phi \in \mathcal{A}\rbrace \ \  a.e.  \  x \in \Pi .  
\end{equation*}
\end{theorem}
\begin{proof}
For $(1)\Rightarrow (2)$ assume that $W$ is a $\Lambda$- modulation invariant  space and $\mathcal{F}$ is the Fourier transform on  $L^{2}(G)$. Then $\mathcal{F}(W)$ is clearly a $\Lambda$- translation invariant subspace of $L^2(\widehat{G})$. By \cite[Theorem 6.5]{BHP}, $\tilde{Z}(W)=Z(\mathcal{F}(W))$ is a multiplicatively invariant subspace of $L^{2}(\Pi , L^{2}(D))$ with respect to the determining set $ \mathcal{D}= \lbrace X_{\lambda} \vert_{\Pi} : \lambda \in \Lambda \rbrace $.      
For $(2)\Rightarrow (3)$, suppose that $\tilde{Z}(W)$ is multiplicatively invariant. By Proposition \ref{p1.1},
\begin{equation*}
\tilde{Z}(W) = \lbrace \phi \in L^{2}(\Pi , L^{2}(D)) \ : \ \phi(x) \in J(x) \ \text{,} \ \  \text{ a.e. }  x \in \Pi \rbrace,
\end{equation*}
for some measurable range function $J$. Applying $\tilde{Z} ^{-1}$ we get
\begin{equation*}
W=\lbrace f \in L^{2}(G) : \tilde{Z}(f)(x)\in J(x) \  \text{,} \ \  \text{for a.e. } \  x \in \Pi \rbrace .
\end{equation*}
Finally, for $(3)\Rightarrow (1)$, assume that \eqref{mi} holds for some measurable range function $J$. For $ f \in W$ and $\lambda \in \Lambda$, using \eqref{Zak} we have
\begin{eqnarray*}
\tilde{Z}(M_{\lambda}f)(x) &=& Z ( \mathcal{F} (M_{\lambda}f)) (x) \\
 &=& Z(T_{\lambda} \hat{f})(x)\\
 &=& X_{\lambda} \vert_{\Pi}(x) Z (\hat{f})(x) \\
 &=& X_{\lambda} \vert_{\Pi}(x) \tilde{Z} (f)(x)  \in J(x) \ \ a.e. \ x \in \Pi ,
\end{eqnarray*}
which proves $(1)$. Now let  $W=M^{\Lambda}(\mathcal{A})$ for some countable subset $\mathcal{A}$ of $L^{2}(G)$. Put $ \mathcal{B}= \tilde{Z}(\mathcal{A})$, then $\mathcal{B}$ is a countable subset of $L^{2}(\Pi , L^{2}(D))$. By \eqref{Zak} 
\begin{equation*}
\tilde{Z}(W) = \overline{span} \lbrace X_{\lambda} (x) \phi(x) \ : \ \lambda \in \Lambda \ , \phi \in \mathcal{B} \rbrace .
\end{equation*}
Now by Proposition \ref{p1.1} the range function associated with $\tilde{Z}(W)$ is
\begin{equation*}
J(x)= \overline{span} \lbrace \Phi(x) \ : \ \Phi \in \mathcal{B} \rbrace = \overline{span}\lbrace \tilde{Z}\phi(x) \ :\ \phi \in \mathcal{A} \rbrace .
\end{equation*}
The proof of one to one correspondence between $\Lambda$- modulation invariant spaces and measurable range functions is similar to \cite[Theorem 6.5]{BHP} and is omitted.
\end{proof}
\begin{corollary}
For an LCA group $G$, a closed subspace $W$ of $L^2(G)$ that is invariant under all modulations (i.e. $W$ is $\widehat{G}$- modulation invariant) can be written as 
\begin{equation} \label{ghatin}
W= \lbrace f \in L^2(G) \ :\ supp \ f \ \subseteq F \rbrace ,
\end{equation} 
for some measurable subset $F \subseteq G$. 
\end{corollary}
\begin{proof}
 In this case $\mathcal{F}(W)$ is a $\widehat{G}$- translation invariant subspace of $L^2(\widehat{G})$. By \cite[Remark 6]{BHP},  $\mathcal{F}(W) = \lbrace  f \in L^2(\widehat{G}) \ : \ supp \ \hat{f} \subseteq E \rbrace  $, for some measurable subset $E \subseteq G$ and \eqref{ghatin} follows.
 \end{proof}
For a countable subset $\mathcal{A} \subseteq L^2(G) $, we can give characterizations of frames and Riesz bases generated by  $\mathcal{A}$ in terms of the operator $\tilde{Z}$ defined in  \eqref{ztild}.
\begin{theorem} \label{fr}
Let $\mathcal{A} \subseteq L^2(G)$ be a countable subset and $J$ be the measurable range function associated with $W= M^{\Lambda}(\mathcal{A})$. Assume that $E^{\Lambda}(\mathcal{A}) := \lbrace M_{\lambda} \phi \ : \ \phi \in \mathcal{A} \rbrace$. The following conditions are equivalent.\\
(1) $E^{\Lambda}(\mathcal{A})$ is a  continuous  frame (continuous Riesz basis) for $W$ with bounds $ 0 < A \leq B < \infty $. \\
(2) The set $\lbrace \tilde{Z} \phi (x) \ : \ \phi \in \mathcal{A} \rbrace$  is a frame (Riesz basis) with bounds $A$ and $B$, for almost every $x \in \Pi$.
\end{theorem}
\begin{proof}
Using the fact that unitary operators preserve frames and Riesz bases \cite[Section 5.3]{Ole}, we know that $E^{\Lambda}(\mathcal{A})$ is a
continuous
 frame for $ M^{\Lambda}(\mathcal{A})$, if and only if
\begin{equation*}
\mathcal{F}(E^{\Lambda}(\mathcal{A})) = \lbrace T_{\lambda} \hat{\phi} \ : \ \phi \in \mathcal{A} \rbrace
\end{equation*}
 is a   continuous  frame (Riesz basis) for
\begin{equation*}
 \mathcal{F}(M^{\Lambda}(\mathcal{A})) = \overline{span} \lbrace T_{\lambda} \hat{\phi} \ : \ \phi \in \mathcal{A} \rbrace .
\end{equation*}
 Equivalently, for almost every $x \in \Pi$, the set  $\lbrace Z \hat{\phi} (x) \ : \ \phi \in \mathcal{A} \rbrace$ is a frame (Riesz basis) for $J(x)$, where $Z$ is as in \eqref{Zak} (see  \cite[Theorem 6.10]{BHP}). Now by \eqref{ztild},
\begin{equation*}
\lbrace Z \hat{\phi} (x) \ : \ \phi \in \mathcal{A} \rbrace = \lbrace \tilde{Z}{\phi} (x) \ : \ \phi \in \mathcal{A} \rbrace .
\end{equation*}
This completes the proof.
\end{proof}
 The following proposition states that every $\Lambda$- modulation invariant space can be decomposed to mutually orthogonal $\Lambda$- modulation invariant spaces each of which is generated by a single function in $L^2(G)$. The proof is similar to \cite[Theorem 5.3]{BR} and so  is omitted.
\begin{proposition} \label{t2.3}
Let $W$ be a $\Lambda$- modulation invariant subspace of $L^{2}(G)$. Then there exist functions $\phi_{n} \in W$, $n \in \mathbb{N}$ such that\\
(1) The set $\{ M_{\lambda} \phi _{n} \ : \ \lambda \in \Lambda  \}$ is a Paseval frame for $M^{\Lambda}(\phi _{n})$. \\
(2) The space $W$ can be decomposed as an orthogonal sum 
\begin{equation*}
W= \bigoplus _{n \in \mathbb{N}} M^{\Lambda}(\phi _{n}).
\end{equation*}
\end{proposition}

\begin{example}
For a fixed prime number $p$, the field of $p$-adic numbers $\mathbb{Q}_p$  is the completion of rational numbers $x= \sum_{j=m}^{\infty}c_{j}p^j$ for $m \in \mathbb{Z}$ and $c_{j} \in \lbrace 0,1, \dots , p-1\rbrace$ under the $p$-adic norm $\vert  .  \vert _p$ defined as follows (see \cite[Chapter 2]{F}). Every nonzero rational $x$ can be uniquely written as  $ x= \frac{r}{s}p^n$, where $r, s, n \in \mathbb{Z}$ and $p$ does not divide $r$ or $s$. We define the $p$-adic norm of $x$ by $ \vert x \vert _p = p^{-n}$, in addition $\vert 0 \vert _p = 0$. Then $\mathbb{Q}_p$ is an additive LCA group and 
 $\mathbb{Z}_p := \lbrace x \in  \mathbb{Q}_p :  \vert x \vert _p  \leq 1  \rbrace =  \left\lbrace \sum_{j=0}^{\infty} c_j p^j : c_j \in \lbrace 0,1, \dots, p-1\rbrace \right\rbrace $ is a closed, compact, and open subgroup of $\mathbb{Q}_p$. A fundamental domain of $\mathbb{Z}_p$ in $\mathbb{Q}_p$ is 
\begin{equation*}
\Omega = \left\lbrace \sum_{j=m}^{-1} c_j p^j : c_j \in \lbrace 0,1, \dots, p-1\rbrace \right\rbrace .
\end{equation*}
 By Theorem \ref{p1.2}, every $\mathbb{Z}_p$- modulation invariant space of $L^2(\mathbb{Q}_p)$ is of the form
\begin{equation*}
W= \lbrace f \in L^2(\mathbb{Q}_p) \ : \ \tilde{Z}f(x) \in J(x) \ a.e. \ x \in \Omega   \rbrace ,
\end{equation*}
for some measurable range function $J: \Omega \longrightarrow \lbrace \textrm{  closed subspaces  of   $L^2(\Omega)$ } \rbrace$, where $\tilde{Z} : L^2(\mathbb{Q}_p) \longrightarrow L^{2}(\Omega , L^{2}(\Omega))$ is the Zak transform defined as \eqref{ztild}.
\end{example}

\section{The Modulation metric}
In \cite{MRK}, we defined a translation metric on the collection of all translation invariant spaces and studied some topological properties of the metric space. In this section, we define a similar metric for the modulation case. Infact  we introduce and investigate topological properties of a modulation metric $\theta$, a metric on the collection of all modulation invariant subspaces of $L^{2}(G)$. Let $MI(G)$ denote the collection of all modulation invariant subspaces of $L^{2}(G)$. For each $V$ and $W$ in $MI(G)$ define 
\begin{equation} \label{tm}
\theta (V,W) = \inf \lbrace \alpha > 0 : m(\lbrace x  \in \Pi : \Vert P_{J_{V} }(x) - P_{J_{W} }(x) \Vert > \alpha \rbrace) = 0\rbrace,
\end{equation}
where $J_{V}$ and $J_{W}$ are the measurable range functions associated with $V$ and $W$, $P_{J_{V}}(x)$ and  $P_{J_{W}}(x)$, $x \in \Pi $, are the orthogonal projections onto $J_{V} (x)$ and $J_{W}(x)$ respectively, $\Vert. \Vert$ denotes the operator norm, and $m$ is the Haar measure of $G$. In the forthcoming proposition, we show that $\theta$ is a metric on $MI(G)$, which is called modulation metric.  Note that if  $V$ and $W$ are modulation invariant spaces, then $\theta (V,W) \leq \epsilon$ if and only if, $\Vert P_{J_{V} }(x) - P_{J_{W} }(x) \Vert \leq \epsilon$, for a.e. $x \in \Pi$.
\begin{proposition} 
With the notation as above, $\theta $ is a metric on $MI(G)$.
\end{proposition}
\begin{proof}
Positivity of $\theta$ follows from the definition. For $V$ and $W$ in $ MI(G)$, if $\theta (V,W) = 0$, one can find a sequence $(\alpha_{n})$ of positive numbers converging to $0$ and a set $E$ of  measure zero such that $\Vert P_{J_{V}} (x) - P_{J_{W}} (x) \Vert \leq \alpha_{n}$, for all $n \in \mathbb{N}$  and for $ x  \in \Pi $. It follows that $ \Vert P_{J_{V}} (x) - P_{J_{W}} (x) \Vert = 0$ for a.e. $ x \in \Pi$, so the projections onto $J_{W}(x)$ and $J_{V}(x)$ are the same a.e. and hence $V = W$, in the sence of the usual convention that two modulation invariant spaces are equal if the corresponding range functions are equal a.e.
On the other hand, $V = W$ implies that $J_{V}(x) = J_{W}(x)$ for a.e. $x \in  \Pi $, which in turn implies that $\Vert P_{J_{V}} (x) - P_{J_{W}} (x) \Vert > 0 $ only on a set of measure $0$. Hence $\theta(V,W) = 0$.
For the triangle inequality, if $U$, $V$, and $W$ are modulation invariant spaces and $\epsilon > 0$, one can get $M_{1},M_{2} > 0$ such that $M_{1} < \theta(V,U) + \frac{\epsilon}{2}$, $M_{2} < \theta(U,W) +\frac{\epsilon}{2}$, $m(\lbrace x \in  \Pi : \Vert P_{J_{V}} (x) - P_{J_{U}} (x) \Vert > M_{1}\rbrace) = 0$, and  $m(\lbrace x \in  \Pi : \Vert P_{J_{U} }(x) - P_{J_{W}} (x) \Vert > M_{2}\rbrace) = 0$.
Applying the triangle inequality for the norm gives
\begin{equation*}
\Vert P_{J_{V}} (x) - P_{J_{W}} (x) \Vert \leq M_{1} + M_{2} \ \ a. e\ x \in \Pi.
\end{equation*}
It follows that $\theta(V,W) \leq \theta(V,U) + \theta(U,W)$. Finally, it's obvious that $\theta(V,U) = \theta(U, V )$.
\end{proof}
In the sequel, we show that  $MI(G)$ is  a complete, noncompact, and disconnected metric space. First of all we need the following lemma.
\begin{lemma}
Let $(J_{n})$ be a sequence of measurable range functions, and let $(P_{n}(x))$ be the corresponding sequence of orthogonal projections onto $J_{n}$'s. Suppose that $(P_{n}(x))$ converges to the orthogonal projection $P(x)$ in the operator norm for $x \in \Pi $. If $J(x)$ is the range of $P(x)$, then $J$ is a measurable range function.
\end{lemma}
\begin{proof}
Let $f \in L^{2}(D)$. Setting $F_{n}(\xi) = P_{n}(\xi)f$ and $F(\xi) = P(\xi)f$, we have 
\begin{equation}
\Vert F_{n}(\xi) - F(\xi) \Vert \leq \Vert P_{n}(\xi) - P(\xi) \Vert \Vert f \Vert.
\end{equation}
It now follows that $F(\xi) = \lim F_{n}(\xi)$. Thus $F$ is
the limit of a sequence $(F_{n})$ of vector valued measurable functions and hence is measurable. That is $J$ is measurable.
\end{proof}
\begin{theorem}\label{complete}\label{th2.3}
The space $MI(G)$ is complete in the modulation metric.
\end{theorem}
\begin{proof}
Suppose $(W_{n})$ is Cauchy in $MI(G)$. Then $(P_{J_{W_{n}}}(x))$ is Cauchy in the Banach space $BL(L^{2}(D))$, the space of all bounded linear operators on $L^{2}(D)$. Hence it converges to an orthogonal projection $P(x)$ for a.e. $x \in \Pi $. Let $J(x)$
be the closed subspace of $L^{2}(D)$ associated with the orthogonal projection $P(x)$. Consider the modulation invariant space $W := \lbrace \varphi \in L^{2}(G): \tilde{Z}\varphi(x) \in J(x) \ a.e.\  x \in \Pi \rbrace$, we have $J_{W}(x) = J(x)$ for a.e. $x \in \Pi $, and hence $P_{J_{W}}(x)= P(x)$ for a.e. $x \in \Pi $. Consequently, $(W_{n})$ converges to $W$ in the modulation metric.
\end{proof}
As a consequence of Theorem \ref{th2.3} we have the following corollary. Let $PMI(G)$ denote the collection of all principal modulation invariant subspaces of $L^2(G)$.
 \begin{corollary}
The space $PMI(G)$ is complete in the modulation metric.
\end{corollary}
\begin{proof}
Suppose that $(W_{n})$ is a Cauchy sequence in $PMI(G)$. By Theorem \ref{complete}, $(W_{n})$ converges to some $W \in MI(G)$. We need only to show that $W$ has a single generator.
For $0 < \epsilon < 1$, choose $p \in \mathbb{N}$ such that $ \theta(W_{n}, W ) < \epsilon$ for all $n \geq p$. This implies that
$\Vert P_{J_{W_{n}} }(x) - P_{J_{W} }(x) \Vert < \epsilon $ for a.e. $x$,   whenever $ n \geq p$. Hence $\dim J_{W} (x) = \dim J_{W_{n}}(x) = 1$
for a.e. $x$ (\cite[Theorem 4.35]{W}). This proves that $W$ can be generated by a single function, and hence $W \in PMI(G)$.
\end{proof}
Let $FMI(G)$ be the collection of all modulation invariant spaces generated by a fixed number of elements of $L^2(G)$. With the  same proof as Corollary 2.4, one can see the following corollary.
 \begin{corollary}
The collection $FMI(G)$ is complete in the modulation metric.
\end{corollary}

Now we show that $MI(G)$ is not  compact.
\begin{proposition}
The space $MI(G)$ is not compact in the modulation metric topology.
\end{proposition}
\begin{proof}
Using \cite[Theorem 45.1]{Mun}, it is enough to show that $MI(G)$ is not totally bounded in the modulation metric. First
choose a countable basis $\lbrace \varphi_{1}, \varphi_{2},\dots \rbrace$ for $L^{2}(G)$. Set $W_{m} = M^{\Lambda}(\mathcal{A}_{m})$, where $\mathcal{A}_{m} = \lbrace \varphi_{1}, \varphi_{2},\dots, \varphi_{m} \rbrace$. Then $W_{m} \subset W_{m+1}$ for any $m$, and hence $\Vert P_{J_{W_{m}}}(x) - P_{J_{W_{m+1}} }(x) \Vert = 1$ for all $x \in \Pi$ (\cite[Theoram 4.30]{W}).
That is $\theta(W_{m}, W_{m+1}) = 1$ for all $m$. Hence for $\epsilon=\frac{1}{2}$, no finite collection of $\epsilon$-balls can contain all $V_{m}$'s.
\end{proof}

 In the next theorem we show that the metric space $MI(G)$ is disconnected.
 \begin{theorem}
The space $MI(G)$ is disconnected in the modulation metric.
\end{theorem}
\begin{proof}
It is enough to show that $MI(G)$ has an open and closed proper subset. That $PMI(G)$ is closed follows from Corollary 2.4. Now we show that it is open. Let $W \in PMI(G)$; put $r=\frac{1}{2}$. We show that
$B_{r}(V) \subseteq PMI(G)$, where $B_{r}(W)$ is an open ball with center $W$ and radius $r$.  Let $V \in B_{r}(V)$; then $\theta(W,V)< \frac{1}{2}$, and hence $\dim J_{W}(x)=\dim J_{V}(x)$ for a.e. $x \in \Pi$ (\cite[Theorem 4.35]{W}). Hence  $V\in PMI(G)$. Since $W$ is arbitrary, then $PMI(G)$ is  an open subspace of $MI(G)$. That is $MI(G)$ is disconnected. 
\end{proof}

\begin{remark}
In the case that the closed subgroup $\Lambda$ of $\widehat{G}$ is also cocompact, our results in this paper can be phrased in terms of the so called  fiberization map $\mathcal{T} : L^2(\widehat{G}) \longrightarrow L^2(\Pi , l^2(\Lambda ^{*})), \mathcal{T} f (x) = \lbrace \hat{f}(x+k)\rbrace _{k \in \Lambda ^*}$, instead of the Zak transform $Z$. 
In this case, Theorem \ref{p1.2} says that a closed subspace $W$ of $L^2(G)$ is $\Lambda$- modulation invariant if and only if $W=\lbrace f \in L^{2}(G) : \tilde{\mathcal{T}}(f)(x)\in J(x)  \text{,} \ \  \text{for a.e. }  x \in \Pi \rbrace$, for some measurable range function 
\begin{equation*}
J: \Pi \longrightarrow \lbrace  closed \ subspaces \  of \   L^{2}(D) \rbrace,
\end{equation*}
in which $\tilde{\mathcal{T}} : L^2(G) \longrightarrow L^2(\Pi , l^2(\Lambda ^{*})), \tilde{\mathcal{T}} := \mathcal{T} \ o \ \mathcal{F}$, and $\mathcal{F}$ is the Fourier transform on $L^2(G)$.
Moreover if $W=M^{\Lambda}(\mathcal{A})$, for some countable subset $\mathcal{A}$ of $L^{2}(G)$, the measurable range function $J$ is given by
\begin{equation*}
J(x) = \overline{span} \lbrace \tilde{\mathcal{T}}(\phi) (x) : \phi \in \mathcal{A}\rbrace \ \  a.e.  \  x \in \Pi ,  
\end{equation*}
 (compare with \cite[Theorem 3.8]{BR}).
\end{remark}

\begin{example}
We give an example of a sequence of $\mathbb{Z}$- modulation invariant subspaces of $L^2(\mathbb{R})$ converging in the modulation metric. Define $ \phi := \textbf{1} _{(0,1)} \in L^2(\mathbb{R})$, and suppose $(\phi _n)$ is the sequence defined by $\phi_{n} (x) = \frac{n+1}{n} \textbf{1}_{(0,1)}(x)$. Let $W= M^{\mathbb{Z}} (\phi)$ and $W_{n} = M^{\mathbb{Z}}(\phi _{n})$. A direct calculation from \cite[Theorem 3.8]{BR} shows that $J_{V}(\xi) = span \lbrace e_0 \rbrace $ and $J_{V_{n}}(\xi) = span \lbrace \frac{n+1}{n} e_0 \rbrace $, where $(e_k)$ is the standard basis of $l^2(\Lambda ^{*})$.  So the projections onto $J_{W_{n}}(\xi)$'s and $J_{W}(\xi)$ are the same and we can conclude that $(W_{n})$ converges to $W$ in the modulation metric.
\end{example}

\bibliographystyle{amsplain}

\end{document}